\documentclass[11pt]{amsart}

\usepackage[T1]{fontenc}
\usepackage{lmodern}
\usepackage{microtype}
\usepackage{amsmath,amssymb,mathtools}
\usepackage{enumitem}
\usepackage{hyperref}
\hypersetup{
  colorlinks=true,
  linkcolor=blue,
  citecolor=blue,
  urlcolor=blue
}

\newcommand{\Q}{\mathbb{Q}}
\newcommand{\Z}{\mathbb{Z}}
\newcommand{\Pui}{\mathcal{P}}
\newcommand{\TrZ}{\operatorname{Tr}_{\Z}}
\newcommand{\den}{\operatorname{den}}
\newcommand{\lcm}{\operatorname{lcm}}
\newcommand{\T}{\mathcal{T}}
\newcommand{\qbinom}[2]{\genfrac{[}{]}{0pt}{}{#1}{#2}_{q}}
\newcommand{\coeff}[2]{[q^{#1}]#2}
\newcommand{\geqcoef}{\succeq_q}

\theoremstyle{plain}
\newtheorem{theorem}{Theorem}[section]
\newtheorem{corollary}[theorem]{Corollary}
\newtheorem{proposition}[theorem]{Proposition}
\newtheorem{lemma}[theorem]{Lemma}
\newtheorem{conjecture}[theorem]{Conjecture}

\theoremstyle{definition}

\newtheorem{example}[theorem]{Example}

\title[On a conjecture of Han and Xiong]
{On a conjecture of Han and Xiong for fractional Gaussian binomial coefficients}

\author{Ken Ono}
\address{Axiom Math, 124 University Avenue, Palo Alto, CA 94301}
\email{ken@axiommath.ai}

\date{July 2026}


\begin{document}

\begin{abstract}
Han and Xiong recently extended the Gaussian binomial coefficient
$\qbinom{r+k}{k}$ to positive rational $r$ and conjectured that its ``integer
trace'', the integer-exponent part of the resulting power
series, is coefficientwise largest at $r=1/2$.  We prove a support-dominance
theorem comparing rational parameters under an explicit divisibility
condition.  It settles the conjecture for every $r\geq 1/2$ and reduces the
full conjecture to the unit fractions $r=\frac{1}{2m}$, only finitely many of which
are nontrivial for each fixed $k$.  A computer  computation
then verifies the conjecture for every positive rational $r$ and every
$k\leq 200$.  The theoretical results were autonomously produced and verified in Lean by AxiomProver.
\end{abstract}

\maketitle

\section{Introduction and Statement of Results}\label{sec:intro}

The binomial coefficient $\binom{N}{k}$ counts the $k$-element subsets of an
$N$-element set.  Its classical $q$-analogue is defined using the
\emph{$q$-shifted factorial}
\[
  (z;q)_n:=\prod_{j=0}^{n-1}(1-zq^j),
  \qquad (z;q)_0:=1.
\]
For integers $N\geq k\geq 0$, the \emph{Gaussian binomial coefficient} (or
\emph{$q$-binomial coefficient}) is
\begin{equation}\label{eq:qbinom-def}
  \qbinom{N}{k}
  :=\frac{(q;q)_N}{(q;q)_k(q;q)_{N-k}}
  =\prod_{j=1}^{k}\frac{1-q^{N-k+j}}{1-q^{j}}.
\end{equation}
Despite its appearance as a rational function, $\qbinom{N}{k}$ is a polynomial
in $q$ with nonnegative integer coefficients, and it relates to the classical
objects in several standard ways.  First, it recovers
$\binom{N}{k}$ as a limiting case: each factor
$(1-q^{N-k+j})/(1-q^{j})$ in \eqref{eq:qbinom-def} tends to $(N-k+j)/j$ as
$q\to 1$, whence
\[
  \lim_{q\to 1}\qbinom{N}{k}=\binom{N}{k}.
\]
Second, it satisfies the finite \emph{$q$-binomial theorem}
\begin{equation}\label{eq:finite-q-binomial}
  (z;q)_k
  =\sum_{s=0}^{k}(-1)^s q^{\binom{s}{2}}\qbinom{k}{s}z^s,
\end{equation}
which specializes at $q=1$ to the classical binomial theorem
\[
  (1-z)^k=\sum_{s=0}^{k}(-1)^s\binom{k}{s}z^s
\]
(for example, see
\cite[Theorem~3.3]{Andrews} or \cite[Chapter~1]{GasperRahman}).
Combinatorially, $\qbinom{N}{k}$ is the rank-generating function for integer
partitions whose Ferrers diagrams fit inside a $k\times(N-k)$ rectangle
(for example, see
\cite[Chapter~3]{Andrews} and \cite[Section~1.7]{Stanley}).

Han and Xiong \cite{HanXiong} recently asked what becomes of these objects
when the index is allowed to be \emph{fractional}.  For a positive rational
number $r$ and an integer $k\geq 0$, they define
\begin{equation}\label{eq:generalized-qbinom}
  G_{r,k}(q)
  :=\qbinom{r+k}{k}
  :=\frac{(q^{r+1};q)_k}{(q;q)_k},
\end{equation}
which agrees with \eqref{eq:qbinom-def} when $r$ is a positive integer (take
$N=r+k$).  If $r=a/b$ is nonintegral and in lowest terms, then $G_{r,k}(q)$ is
naturally a formal power series in the fractional power $q^{1/b}$.  Han and
Xiong extract from it an ordinary power series by discarding every term whose
exponent is not an integer.  Writing $\TrZ$ for this \emph{integer trace}
(made precise in Section~\ref{sec:trace-criteria}), they set
\begin{equation}
  \T_{r,k}(q):=\TrZ\bigl(G_{r,k}(q)\bigr)\in\Z[[q]].
\end{equation}
When $r$ is a positive integer, every exponent occurring in $G_{r,k}$ is
already an integer, nothing is discarded, and $\T_{r,k}=G_{r,k}=\qbinom{r+k}{k}$
is the ordinary Gaussian binomial polynomial.  Thus taking the trace is
nontrivial only when $r\notin\Z$.  In that case,
Proposition~\ref{prop:trace-formula} shows that the finite $q$-binomial
numerator retains only the summands whose indices are divisible by the
reduced denominator of $r$, often leaving a sparse subseries; the resulting
trace is no longer given by an evident product formula.
Example~\ref{ex:k4} carries out the extraction for three values of $r$.
To compare such series coefficientwise, for $F,G\in\Q[[q]]$ we write
\[
  F\geqcoef G
  \quad\Longleftrightarrow\quad
  \coeff{n}{(F-G)}\geq 0\quad\text{for every }n\geq 0.
\]
This fractional-index framework is due to \cite{HanXiong}.  We have
changed only the notation, writing $\T_{r,k}$ for their
$\bigl[\begin{smallmatrix}r+k\\k\end{smallmatrix}\bigr]_q\big|_{\Z}$ and
$\geqcoef$ for their $\geq_q$.  The central conjecture of
\cite[Conjecture 2.2]{HanXiong} asserts that, among
all positive rational parameters, $r=1/2$ produces the largest integer trace.

\begin{conjecture}[Han--Xiong]\label{conj:HX}
For every positive rational number $r$ and every integer $k\geq 1$, we have
\begin{equation}\label{eq:HX-conjecture}
  \T_{\frac{1}{2},k}(q)\geqcoef \T_{r,k}(q).
\end{equation}
\end{conjecture}

The following example illustrates the conjecture in the smallest case that
our results identify as genuinely nontrivial.

\begin{example}\label{ex:k4}
Take $k=4$, and let us extract the trace for $r=1/4$ from the definition.  Here
$z=q^{r+1}=q^{5/4}$, and \eqref{eq:finite-q-binomial} gives
\[
  \bigl(q^{5/4};q\bigr)_4
  =1
  -q^{5/4}\qbinom{4}{1}
  +q^{7/2}\qbinom{4}{2}
  -q^{27/4}\qbinom{4}{3}
  +q^{11}.
\]
Each $\qbinom{4}{s}$ is a polynomial in $q$ with integer exponents, so all the
exponents contributed by the $s$-th term lie in $E+\Z_{\geq 0}$, where $E$ is
the exponent displayed in front of it.  Consequently the $s$-th term
contributes to the integer trace precisely when that displayed exponent is an
integer.  Of the five values $0,\ \tfrac54,\ \tfrac72,\ \tfrac{27}4,\ 11$ only
the first and last qualify, and since $1/(q;q)_4$ has integer exponents we may
extract the trace of the numerator first and divide afterwards:
\[
  \T_{\frac{1}{4},4}(q)=\frac{1+q^{11}}{(q;q)_4}
  =1+q+2q^2+3q^3+5q^4+6q^5+9q^6+11q^7+15q^8+\cdots.
\]
Note that three of the five terms of the numerator, including the entire
contribution of $\qbinom{4}{2}$, are discarded.

The same computation for $r=1/2$ has $z=q^{3/2}$ and exponents
$E=\tfrac{s(s+2)}2$, that is $0,\ \tfrac32,\ 4,\ \tfrac{15}2,\ 12$.  Now the
terms $s=0,2,4$ survive, and each carries the sign $(-1)^s=+1$, so
\begin{align*}
  \T_{\frac{1}{2},4}(q)
  &=\frac{1+q^{4}\qbinom{4}{2}+q^{12}}{(q;q)_4}
   =\frac{1+q^{12}}{(q;q)_4}+\frac{q^{4}}{(q;q)_2^{\,2}}\\
  &=1+q+2q^2+3q^3+6q^4+8q^5+14q^6+19q^7+29q^8+\cdots,
\end{align*}
using $\qbinom{4}{2}/(q;q)_4=1/(q;q)_2^{\,2}$.  Subtracting,
\[
  \T_{\frac{1}{2},4}(q)-\T_{\frac{1}{4},4}(q)
  =\frac{q^4\qbinom{4}{2}+q^{12}-q^{11}}{(q;q)_4}
  =q^4+2q^5+5q^6+8q^7+14q^8+\cdots
  \geqcoef 0,
\]
in agreement with \eqref{eq:HX-conjecture}.  The inequality is not termwise
obvious: the numerator contains $-q^{11}$, so positivity is
produced only after division by $(q;q)_4$.  That numerator is the polynomial
$D_{2,4}$ of Section~\ref{subsec:cumulative}, and this is the smallest case in
which the phenomenon occurs.

Finally, $r=1/3$ shows what an odd denominator does.  Here the exponents are
$0,\ \tfrac43,\ \tfrac{11}3,\ 7,\ \tfrac{34}3$, so only $s=0$ and $s=3$
survive.  But $s=3$ is odd, so its sign is $(-1)^3=-1$ and
\[
  \T_{\frac{1}{3},4}(q)=\frac{1-q^{7}\qbinom{4}{3}}{(q;q)_4}
  =\frac{1}{(q;q)_4}-\frac{q^{7}}{(q;q)_3(q;q)_1}.
\]
This series has $\coeff{19}{\T_{1/3,4}}=-8$.  Thus an integer trace need not
have nonnegative coefficients; the conjecture compares two traces and makes
no separate positivity assertion about either one.
\end{example}

Han and Xiong established Conjecture~\ref{conj:HX} in several cases.  They
proved it for every positive rational $r$ when $k\leq 3$
\cite[Proposition 4.1]{HanXiong}, and for every $k\geq 1$ in the two infinite
parameter families $r=m+1/2$ ($m\geq 0$ an integer)
\cite[Proposition 5.1]{HanXiong} and $r=m$ ($m\geq 1$ an integer)
\cite[Proposition 5.2]{HanXiong}.  Beyond these families, they verified
the two parameters $r=1/3$ and $r=1/4$ for all $1\leq k\leq 150$ by exact
symbolic computation \cite[Propositions 6.1 and 6.3]{HanXiong}.

In this paper we prove a comparison principle for pairs of rational
parameters whose positive summand-index sets satisfy an explicit divisibility
condition.  The two infinite families above follow directly from this
principle, and the cases $k\leq 3$ follow after the canonical reduction below.
The earlier finite computations for $r=1/3$ and $r=1/4$ are instead subsumed
by our exact verification through $k=200$.  For a positive rational number
$x$, let $\den(x)$ denote the denominator of $x$ in lowest terms.  Our main
theorem is the following.

\begin{theorem}[Support dominance]\label{thm:support-dominance}
Let $u=c/d$ and $r=a/b$ be positive rational numbers in lowest terms, and
suppose that
\begin{equation}\label{eq:dominance-hypotheses}
  d\ \text{is even},
  \qquad d\mid \lcm(b,2),
  \qquad u\leq r.
\end{equation}
Then, for every integer $k\geq 1$, we have
\[
  \T_{u,k}(q)\geqcoef \T_{r,k}(q).
\]
\end{theorem}

Theorem~\ref{thm:support-dominance} includes monotonicity for each fixed
even denominator and also compares parameters with different reduced
denominators whenever the divisibility condition in
\eqref{eq:dominance-hypotheses} holds.  Taking $u=1/2$, whose denominator
$d=2$ is even and divides $\lcm(b,2)$ for every $b$, we obtain the conjecture
on the entire half-line $r\geq 1/2$.

\begin{corollary}\label{cor:half-line}
For every rational number $r\geq 1/2$ and every integer $k\geq 1$, we have
\[
  \T_{1/2,k}(q)\geqcoef \T_{r,k}(q).
\]
\end{corollary}

Corollary~\ref{cor:half-line} recovers both infinite families of
\cite{HanXiong} (namely $r=m$ and $r=m+1/2$) and extends them to every
reduced fraction $a/b\geq 1/2$, of arbitrary denominator.
Consequently, Conjecture~\ref{conj:HX} is open only for parameters
$0<r<1/2$, and there, by Theorem~\ref{thm:finite-verification} below, only
for $k>200$.

Theorem~\ref{thm:support-dominance} does more: it collapses that remaining
interval onto a single canonical sequence of parameters.  Define the
\emph{even core} of a positive rational $r$ by
\begin{equation}\label{eq:even-core}
  r^{\flat}:=\frac{1}{\lcm(\den(r),2)},
\end{equation}
so that if $r=a/b$ is reduced, then
\[
 r^{\flat}=
 \begin{cases}
   \frac{1}{b},& b\text{ even},\\
   \frac{1}{2b},& b\text{ odd}.
 \end{cases}
\]
Therefore, $r^{\flat}$ depends only on the denominator of $r$, not on its numerator.
In these terms the reduction takes the following form.

\begin{theorem}[Canonical reduction]\label{thm:canonical-reduction}
For every positive rational number $r$ and every integer $k\geq 1$, we have
\begin{equation}\label{eq:core-dominates}
  \T_{r^{\flat},k}(q)\geqcoef \T_{r,k}(q).
\end{equation}
Consequently, Conjecture~\ref{conj:HX} is equivalent to the family of
inequalities
\begin{equation}\label{eq:unit-family}
  \T_{\frac{1}{2},k}(q)\geqcoef \T_{\frac{1}{2m},k}(q)
  \qquad(m\geq 1,\ k\geq 1).
\end{equation}
Moreover, for each fixed $k$ it suffices to verify \eqref{eq:unit-family} in
the finite range
\begin{equation}\label{eq:finite-unit-family}
  2\leq m\leq \left\lfloor\frac{k}{2}\right\rfloor.
\end{equation}
\end{theorem}

We call the unit fractions $\frac{1}{2m}$ the \emph{primitive parameters} of this
reduction.  Theorem~\ref{thm:canonical-reduction} shows that numerators
greater than one and odd reduced denominators never produce new cases: the
entire two-parameter conjecture, over all of $\Q_{>0}$, is governed by the
single sequence $1/4,\,1/6,\,1/8,\ldots$ (the case $m=1$ being an equality).
For example, a proof of the primitive comparison for $r=1/6$ would
automatically yield the comparison for $r=1/3$, whereas the reverse
implication does not follow from the dominance relation.  The reduction is
purely structural and does not depend on computation.

For each fixed $k$, the reduction leaves only finitely many primitive
inequalities.  Using the sufficient cumulative-coefficient criterion below,
an exact computation verifies all of them for $k\leq 200$, extending the
numerical evidence of \cite{HanXiong} from two parameters through $k=150$ to
\emph{all} parameters through $k=200$.

\begin{theorem}[Finite verification]\label{thm:finite-verification}
For every positive rational number $r$ and every integer $k$ with
$1\leq k\leq 200$,
\[
  \T_{\frac{1}{2},k}(q)\geqcoef \T_{r,k}(q).
\]
\end{theorem}

The proof of Theorem~\ref{thm:finite-verification} is computer-assisted.
The program checks a finite list of coefficient inequalities over the
integers, using exact arithmetic and complete coefficient vectors.  It uses
no floating-point arithmetic, degree truncation, or probabilistic testing.
Theorems~\ref{thm:support-dominance} and
\ref{thm:canonical-reduction} are entirely independent of that calculation.
We stress that we do not prove \eqref{eq:unit-family} uniformly in $m$ and
$k$; in particular, the all $k$ case $r=1/4$ remains open.  The precise
remaining problem is recorded at the end of
Section~\ref{sec:proofs}.

This paper builds directly on \cite{HanXiong}.  The fractional-index
extension \eqref{eq:generalized-qbinom}, the integer trace, the order
$\geqcoef$, the exponent $E_r(s)$ and trace formula, the denominator-parity
observation, the coefficientwise-order properties used in
Lemma~\ref{lem:shift-dominance}, the identity underlying
\eqref{eq:half-explicit}, and
the polynomial $D_{2,k}=H_k$ are all theirs.  
Our new contributions are the cross-parameter comparison in
Theorem~\ref{thm:support-dominance}, the canonical reduction in
Theorem~\ref{thm:canonical-reduction}, the cumulative-coefficient criterion of
Lemma~\ref{lem:cumulative}, and the exact verification for all parameters
through $k=200$.
 Han and Xiong's separate
arguments for the half-integer and positive-integer families do not provide
the cross-denominator comparison needed here.

The proof of Theorem~\ref{thm:support-dominance} starts from the trace formula,
which writes $\T_{a/b,k}$ as a signed sum of series with nonnegative
coefficients indexed by multiples of $b$.  When the comparison parameter $u$
has even denominator, every surviving sign on its side is positive.  The
divisibility hypothesis nests the positive summand-index set for $\T_{r,k}$
inside the summand-index set for $\T_{u,k}$, while $u\leq r$ moves every
matched term to an earlier exponent.  Lemma~\ref{lem:shift-dominance} converts
that shift into coefficientwise dominance.  Taking $u=1/2$ gives the
half-line result, and taking $u=r^{\flat}$ gives the canonical reduction.

The paper is organized as follows.  Section~\ref{sec:trace-criteria}
collects the required preliminaries.  Section~\ref{sec:proofs} proves
Theorems~\ref{thm:support-dominance}, \ref{thm:canonical-reduction}, and
\ref{thm:finite-verification}, together with their corollaries, and closes by
recording the remaining open problem.  Finally, in
Appendix~\ref{sec:AI}, we describe how AxiomProver, an AI system under
development, autonomously found, formalized, and verified in Lean the proofs
of Theorem~\ref{thm:support-dominance}, Corollary~\ref{cor:half-line}, and
Theorem~\ref{thm:canonical-reduction}, and we record precisely what was and
was not supplied to it.

\section*{Acknowledgements} 
\noindent This work served as a test case for AxiomProver, an AI system for Lean proof development currently under active development. We thank the engineers who built it: Srihari Ganesh, Tobias Gessler, Leopold Haller, Vasily Ilin, Albert Jiang, Tadeusz Jordan, Andranik Kurghinyan, Kenny Lau, Simon Mahns, Michał Mogielnicki, Gaurang Pendharkar, Karun Ram, Aditya Ramabadran, Aleksey Tsaplin, and Chenkai Wang. The author also thanks Michał Mogielnicki for managing the repo.

\section{Nuts and bolts}\label{sec:trace-criteria}

This section collects the tools required for the proofs of the results
stated in the introduction: the integer-trace formalism, the exact trace
formula, a shift-dominance lemma, and the cumulative-coefficient criterion
behind the exact verification.

\subsection{The integer trace}\label{subsec:trace}

We recall the formal setting of \cite[Section~1]{HanXiong}.  Let
\[
  \Pui:=\bigcup_{M\geq 1}\Q[[q^{1/M}]]
\]
be the directed union of fractional power series rings (i.e. whose exponents are
nonnegative rationals with a common denominator).  Since the reduced
denominator of $r$ is unbounded as $r$
ranges over $\Q_{>0}$, no single $\Q[[q^{1/M}]]$ contains the whole family
$\{G_{r,k}\}$, and $\Pui$ provides a common ambient ring.  If
\[
  F(q)=\sum_{\alpha\in \frac1M\Z_{\geq 0}} c(\alpha)q^{\alpha},
\]
then its \emph{integer trace} is
\begin{equation}\label{eq:trace-definition}
  \TrZ(F):=\sum_{n\geq 0}c(n)q^n.
\end{equation}
As observed in \cite{HanXiong}, this is the natural projection of $\Pui$ onto
$\Q[[q]]$ retaining the integer-exponent part; it is $\Q$-linear.  It is not a
ring homomorphism, but it commutes with multiplication by an ordinary power
series, which is the only property we need.

\begin{lemma}\label{lem:trace-factor}
If $F\in\Pui$ and $H\in\Q[[q]]$, then
\[
  \TrZ(FH)=\TrZ(F)H.
\]
\end{lemma}

\begin{proof}
Every exponent occurring in $H$ is an integer.  Hence, for an exponent
$\alpha$ occurring in $F$ and an exponent $n$ occurring in $H$, the sum
$\alpha+n$ is an integer if and only if $\alpha$ is an integer.  Explicitly,
if $F=\sum_{\alpha}c(\alpha)q^{\alpha}$ and $H=\sum_{n}h(n)q^{n}$, then for
an integer exponent $n_0\geq 0$ the coefficient of $q^{n_0}$ in $FH$ is
$\sum_{\alpha+n=n_0}c(\alpha)h(n)$, a finite sum since $0\leq n\leq n_0$, and
every $\alpha$ contributing to it is the integer $n_0-n$.  Retaining
the integer-exponent terms before or after multiplication by $H$ therefore
gives the same series.
\end{proof}

\subsection{The exact trace formula}\label{subsec:trace-formula}

Substituting $z=q^{r+1}$ into the $q$-binomial theorem
\eqref{eq:finite-q-binomial} gives, as in \cite[eqs.\ (6)--(7)]{HanXiong},
\begin{equation}\label{eq:numerator-expansion}
  (q^{r+1};q)_k
  =\sum_{s=0}^{k}(-1)^s q^{E_r(s)}\qbinom{k}{s},
  \qquad
  E_r(s):=\binom{s+1}{2}+rs.
\end{equation}
Since $\binom{s+1}{2}=\binom{s}{2}+s$, our $E_r(s)$ equals the exponent
$E(s,r)=\binom{s}{2}+(r+1)s$ of \cite[eq.\ (7)]{HanXiong}.
Han and Xiong call $\{s:0\leq s\leq k,\ E_r(s)\in\Z\}$ the \emph{integer
support} of $\T_{r,k}$, since only these indices contribute integer powers of
$q$ to the numerator.  The following proposition is a restatement of their
analysis in \cite[Section~2]{HanXiong}, recorded here in the form used below;
we include the short proof to keep the paper self-contained.

\begin{proposition}[Trace formula]\label{prop:trace-formula}
Let $r=a/b$ be a positive rational number in lowest terms.  For $k\geq 1$,
\begin{align}
  \T_{a/b,k}(q)
  &=\sum_{\substack{0\leq s\leq k\\ b\mid s}}
      (-1)^s q^{E_{a/b}(s)}
      \frac{1}{(q;q)_s(q;q)_{k-s}}
      \label{eq:trace-formula-s}\\
  &=\sum_{j=0}^{\lfloor k/b\rfloor}
      (-1)^{bj}
      \frac{q^{\binom{bj+1}{2}+aj}}
           {(q;q)_{bj}(q;q)_{k-bj}}.
      \label{eq:trace-formula-j}
\end{align}
\end{proposition}

\begin{proof}
The integer $\binom{s+1}{2}$ contributes no fractional part to $E_{a/b}(s)$, so
$E_{a/b}(s)\in\Z$ if and only if $as/b\in\Z$, that is, if and only if
$b\mid as$; and since $\gcd(a,b)=1$ this holds exactly when $b\mid s$.
Lemma~\ref{lem:trace-factor} and
\eqref{eq:numerator-expansion} therefore give
\[
  \T_{a/b,k}(q)
  =\frac{1}{(q;q)_k}
    \sum_{\substack{0\leq s\leq k\\b\mid s}}
      (-1)^s q^{E_{a/b}(s)}\qbinom{k}{s}.
\]
Finally,
\[
  \frac{1}{(q;q)_k}\qbinom{k}{s}
  =\frac{1}{(q;q)_s(q;q)_{k-s}},
\]
which proves \eqref{eq:trace-formula-s}; setting $s=bj$ gives
\eqref{eq:trace-formula-j}.
\end{proof}

For later use, set
\begin{equation}\label{eq:Phi-definition}
  \Phi_{k,s}(q):=\frac{1}{(q;q)_s(q;q)_{k-s}}
  \qquad(0\leq s\leq k).
\end{equation}
Each $\Phi_{k,s}(q)$ has nonnegative integer coefficients.  Indeed, the
classical expansion $1/(q;q)_m=\sum_{n\geq0}p_m(n)q^n$, in which $p_m(n)$
counts the partitions of $n$ into parts of size at most $m$
\cite[Theorem~1.1]{Andrews}, exhibits $\Phi_{k,s}$ as a product of two such
series.  In the language
of Proposition~\ref{prop:trace-formula}, the summand indexed by $s$ occurs if
and only if $b\mid s$, and its sign is positive exactly when $s$ is even.
Hence the \emph{positive summand-index set} of $\T_{r,k}$ is
\begin{equation}\label{eq:positive-index-set}
  \{s:0\leq s\leq k,\ \lcm(b,2)\mid s\},
\end{equation}
and the \emph{negative summand-index set} is
\begin{equation}\label{eq:negative-index-set}
  \{s:0\leq s\leq k,\ b\mid s,\ s\text{ odd}\}.
\end{equation}
In particular, the negative summand-index set is empty when $b$ is even.  This
parity
dichotomy is observed in \cite[Section~2]{HanXiong}, where it is the
motivation for singling out $r=1/2$; the point of
Theorem~\ref{thm:support-dominance} is that it also drives a comparison
between two different parameters.

\subsection{A shift-dominance lemma}\label{subsec:shift}

The key point is that moving a monomial to a later integral exponent can only
decrease its product with $\Phi_{k,s}(q)$ in coefficientwise order.

\begin{lemma}[Shift dominance]\label{lem:shift-dominance}
Let $k\geq 1$, let $0\leq s\leq k$, and let $A$ and $B$ be nonnegative
integers with $A\leq B$.  Then
\begin{equation}\label{eq:shift-dominance}
  q^A\Phi_{k,s}(q)\geqcoef q^B\Phi_{k,s}(q).
\end{equation}
\end{lemma}

\begin{proof}
The claim is immediate when $A=B$.  Suppose that $A<B$.  Since $k\geq 1$, at
least one of $s$ and $k-s$ is positive.  Thus the denominator in
\eqref{eq:Phi-definition} contains at least one factor $1-q$; cancelling one
such factor exhibits $(1-q)\Phi_{k,s}(q)$ as a finite product of
geometric-series factors $(1-q^e)^{-1}$, whence
\begin{equation}\label{eq:Psi-positive}
  (1-q)\Phi_{k,s}(q)\in\Z_{\geq 0}[[q]].
\end{equation}
Using $1-q^{B-A}=(1-q)(1+q+\cdots+q^{B-A-1})$, we obtain
\[
  \bigl(q^A-q^B\bigr)\Phi_{k,s}(q)
  =q^A(1+q+\cdots+q^{B-A-1})(1-q)\Phi_{k,s}(q).
\]
Every factor on the right has nonnegative coefficients, proving
\eqref{eq:shift-dominance}.
\end{proof}

The condition $k\geq 1$ is needed only to guarantee a factor $1-q$ in the
denominator, and every comparison in Conjecture~\ref{conj:HX} has $k\geq 1$.
Lemma~\ref{lem:shift-dominance} combines two properties of the coefficientwise
order established in \cite[Lemma 3.1]{HanXiong}, namely that
$(q^{A}-q^{B})/(1-q)\geqcoef 0$ for $A\leq B$ and that division by factors
$1-q^{j}$ preserves the order; we have stated it in the single form needed
here.

\subsection{A cumulative-coefficient criterion}\label{subsec:cumulative}

The exact verification in Section~\ref{sec:proofs} rests on the following
sufficient criterion, which converts a coefficientwise inequality between two
traces into a \emph{finite} list of integer inequalities.  For $m\geq 1$,
Proposition~\ref{prop:trace-formula} applied to $r=\frac{1}{2m}$ (so $a=1$,
$b=2m$, every sign is positive, and
$E_{1/(2m)}(2m\ell)=\binom{2m\ell+1}{2}+\ell=2m^2\ell^2+(m+1)\ell$) gives
\begin{equation}\label{eq:A-m-k-definition}
  \T_{\frac{1}{2m},k}(q)=\frac{A_{m,k}(q)}{(q;q)_k},
\end{equation}
where
\begin{equation}\label{eq:A-m-k}
  A_{m,k}(q)
  :=\sum_{\ell=0}^{\lfloor k/(2m)\rfloor}
      q^{2m^2\ell^2+(m+1)\ell}
      \qbinom{k}{2m\ell}.
\end{equation}
For $m=1$, this is exactly the numerator for $r=1/2$ computed in
\cite[eq.\ (19)]{HanXiong}.  Put
\[
  D_{m,k}(q):=A_{1,k}(q)-A_{m,k}(q).
\]
The family $D_{m,k}$ directly generalizes the numerator polynomial
$H_k$ of \cite[eq.\ (24)]{HanXiong}: since $2m^2\ell^2+(m+1)\ell$ specializes
at $m=2$ to $\ell(8\ell+3)$ and at $m=1$ to $2\ell(\ell+1)$, and the terms
$\ell=0$ cancel in the difference, one has
\begin{equation}\label{eq:D2-is-H}
  D_{2,k}(q)=H_k(q)
  \qquad(k\geq 1).
\end{equation}
The case $m=2$ of the criterion below is therefore a statement about their
polynomial, and \cite[Conjecture 6.2]{HanXiong}, which predicts that
$H_k(q)\geqcoef 0$ for $k\geq 19$, would imply it in that range.

\begin{lemma}[Cumulative-coefficient criterion]\label{lem:cumulative}
Let $k\geq 1$ and $m\geq 1$.  If
\begin{equation}\label{eq:cumulative-criterion}
  \frac{D_{m,k}(q)}{1-q}\in\Z_{\geq 0}[[q]],
\end{equation}
then
\[
  \T_{\frac{1}{2},k}(q)\geqcoef \T_{\frac{1}{2m},k}(q).
\]
Moreover, if $N=\deg D_{m,k}$ and $D_{m,k}(q)=\sum_{n=0}^{N}d_nq^n$, then
\eqref{eq:cumulative-criterion} is equivalent to the finite set of
inequalities
\begin{equation}\label{eq:cumulative-sums}
  \sum_{n=0}^{j}d_n\geq 0
  \qquad(0\leq j\leq N).
\end{equation}
\end{lemma}

\begin{proof}
Since $(q;q)_k=(1-q)(q^2;q)_{k-1}$, we have
\[
  \T_{\frac{1}{2},k}(q)-\T_{\frac{1}{2m},k}(q)
  =\frac{D_{m,k}(q)/(1-q)}{(q^2;q)_{k-1}}.
\]
For $k=1$, the reciprocal of $(q^2;q)_{k-1}$ is $1$.  For $k\geq 2$, it
has nonnegative integer coefficients and is the generating function for
partitions into parts from $\{2,3,\ldots,k\}$.  Thus
\eqref{eq:cumulative-criterion} implies the desired inequality.

Since $1/(1-q)=1+q+q^2+\cdots$, the coefficient of $q^j$ in
$D_{m,k}(q)/(1-q)$ is $\sum_{n=0}^{\min(j,N)}d_n$.  For $j>N$, this value is
the constant total sum
$\sum_{n=0}^{N}d_n$, which is already the case $j=N$ in
\eqref{eq:cumulative-sums}.  Hence the infinite coefficientwise condition is
equivalent to the displayed finite list.
\end{proof}

Criterion \eqref{eq:cumulative-sums} is a systematic form of the pairing
strategy of \cite[Propositions 6.1 and 6.3]{HanXiong}.  In the relevant
numerator polynomials---their $H_k$ for $r=1/4$ and the analogous polynomial
for $r=1/3$---each negative term $-c_2q^{d_2}$ is matched with an earlier
positive term $c_1q^{d_1}$ satisfying $c_1\geq c_2>0$, so that
$(c_1q^{d_1}-c_2q^{d_2})/(1-q)\geqcoef 0$.  To see the equivalence with
partial sums, first collect like powers and split each nonzero coefficient
into signed unit monomials.  A greedy scan from low to high degree can then
pair every negative unit with an unmatched positive unit at a smaller degree
exactly when every running sum is nonnegative.  The criterion is thus
checked by one pass over the coefficients rather than by a search for a
pairing.

\section{Proofs of the Main Results}\label{sec:proofs}

We now prove the results stated in Section~\ref{sec:intro}, in order:
Theorem~\ref{thm:support-dominance} and Corollary~\ref{cor:half-line}
(Section~\ref{subsec:proof-dominance}),
Theorem~\ref{thm:canonical-reduction}
(Section~\ref{subsec:proof-reduction}), and
Theorem~\ref{thm:finite-verification}
(Section~\ref{subsec:proof-verification}).  We close by recording the
remaining open problem.

\subsection{Proof of the support-dominance theorem}
\label{subsec:proof-dominance}

The proof compares the two trace formulas supplied by
Proposition~\ref{prop:trace-formula} summand by summand.  Recall the setting:
$u=c/d$ and $r=a/b$ are reduced, $d$ is even, $d\mid\lcm(b,2)$, and $u\leq r$.
The first hypothesis makes every surviving summand of $\T_{u,k}$ positive; the
second embeds the positive summand-index set of $\T_{r,k}$ into the
summand-index set of $\T_{u,k}$; the third places each matched summand of
$\T_{u,k}$ at an exponent no larger than that of its partner.
Lemma~\ref{lem:shift-dominance} converts those exponent inequalities into
coefficientwise dominance, and the summands left unmatched contribute with a
favorable sign.

\begin{proof}[Proof of Theorem \ref{thm:support-dominance}]
Write $u=c/d$ and $r=a/b$ in lowest terms, put $L=\lcm(b,2)$, and abbreviate
the three relevant index sets by
\begin{align*}
  U&:=\{s: 0\leq s\leq k,\ d\mid s\},\\
  R_{+}&:=\{s: 0\leq s\leq k,\ L\mid s\},\\
  R_{-}&:=\{s: 0\leq s\leq k,\ b\mid s,\ s\text{ odd}\},
\end{align*}
so that $U$ is the summand-index set of $\T_{u,k}$, while $R_{+}$ and
$R_{-}$ are, respectively, the positive and negative summand-index sets
\eqref{eq:positive-index-set} and \eqref{eq:negative-index-set} for $\T_{r,k}$.
Since $d$
is even, every element of $U$ is even, so by
\eqref{eq:positive-index-set} and \eqref{eq:negative-index-set} every surviving
sign is positive.  Proposition~\ref{prop:trace-formula} therefore gives
\begin{equation}\label{eq:u-positive-expansion}
  \T_{u,k}(q)
  =\sum_{s\in U}
      q^{E_u(s)}\Phi_{k,s}(q).
\end{equation}
On the other hand, separating the positive and negative summand-index sets for
$\T_{r,k}$,
\begin{equation}\label{eq:r-sign-expansion}
  \T_{r,k}(q)
  =\sum_{s\in R_{+}}
      q^{E_r(s)}\Phi_{k,s}(q)
   -\sum_{s\in R_{-}}
      q^{E_r(s)}\Phi_{k,s}(q).
\end{equation}
Because $d\mid L$, any $s$ divisible by $L$ is divisible by $d$; that is,
\[
  R_{+}\subseteq U .
\]
Splitting the sum \eqref{eq:u-positive-expansion} as
$U=R_{+}\sqcup(U\setminus R_{+})$ and subtracting
\eqref{eq:r-sign-expansion} gives the exact identity
\begin{align}
  \T_{u,k}(q)-\T_{r,k}(q)
  ={}&\sum_{s\in R_{+}}
       \bigl(q^{E_u(s)}-q^{E_r(s)}\bigr)\Phi_{k,s}(q)
       \label{eq:dominance-decomposition}\\
   &+\sum_{s\in U\setminus R_{+}}
       q^{E_u(s)}\Phi_{k,s}(q)
    +\sum_{s\in R_{-}}
       q^{E_r(s)}\Phi_{k,s}(q).\notag
\end{align}
The last two sums have nonnegative coefficients; note that the terms of
$R_{-}$, which enter $\T_{r,k}$ negatively, therefore enter the
\emph{difference} positively, and so help rather than hinder.  For an index
$s\in R_{+}$, both exponents $E_u(s)$ and $E_r(s)$
are nonnegative integers (since $d\mid L\mid s$ and $b\mid L\mid s$), and the
hypothesis $u\leq r$ gives
\[
  E_u(s)=\binom{s+1}{2}+us
  \leq \binom{s+1}{2}+rs=E_r(s).
\]
Lemma~\ref{lem:shift-dominance} shows that every summand in the first sum of
\eqref{eq:dominance-decomposition} also has nonnegative coefficients.  Hence
$\T_{u,k}(q)-\T_{r,k}(q)\in\Z_{\geq 0}[[q]]$, as required.
\end{proof}

Each hypothesis in \eqref{eq:dominance-hypotheses} was used at exactly one
point of the argument.  The divisibility condition cannot
simply be dropped: the pair $u=1/4$, $r=1/2$ satisfies the other two
hypotheses but not $d\mid\lcm(b,2)$, and Example~\ref{ex:k4} shows that
$\T_{\frac{1}{4},4}(q)\not\geqcoef\T_{\frac{1}{2},4}(q)$.

\begin{proof}[Proof of Corollary \ref{cor:half-line}]
Apply Theorem~\ref{thm:support-dominance} with $u=1/2$.  Its denominator is
$d=2$, which is even and divides $\lcm(b,2)$ for every positive integer $b$.
The remaining hypothesis is precisely $1/2\leq r$.
\end{proof}

Theorem~\ref{thm:support-dominance} also propagates any established
even-denominator case to a larger set of parameters.

\begin{corollary}[Propagation]\label{cor:propagation}
Let $d$ be even.  Suppose, for a fixed $k$, that
\[
  \T_{\frac{1}{2},k}(q)\geqcoef \T_{c/d,k}(q)
\]
for some reduced positive rational $c/d$.  Then
\[
  \T_{\frac{1}{2},k}(q)\geqcoef \T_{a/b,k}(q)
\]
for every reduced $a/b\geq c/d$ satisfying $d\mid\lcm(b,2)$.
\end{corollary}

\begin{proof}
Theorem~\ref{thm:support-dominance} gives
$\T_{c/d,k}\geqcoef\T_{a/b,k}$, and coefficientwise order is transitive.
\end{proof}

\subsection{Proof of the canonical reduction}
\label{subsec:proof-reduction}

The even core \eqref{eq:even-core} was designed so that
Theorem~\ref{thm:support-dominance} applies to the pair $(r^{\flat},r)$ with
nothing left to check: its denominator is even by construction, and it is at
most $r$ because its numerator is $1$ and its denominator is a multiple of
that of $r$.  That gives
\eqref{eq:core-dominates} at once, and the equivalence with the unit-family
\eqref{eq:unit-family} follows by composing it with the conjectured
inequality at $r=1/2$.  The finite range \eqref{eq:finite-unit-family} is
then a matter of identifying the two traces being compared when $2m$ exceeds
$k$, at which point only the summand $s=0$ survives on the right.

\begin{proof}[Proof of Theorem \ref{thm:canonical-reduction}]
Let $r=a/b$ be in lowest terms and put $L=\lcm(b,2)$.  The rational number
$r^{\flat}=1/L$ is reduced and has even denominator $L$.  Moreover,
\[
  r^{\flat}\leq r.
\]
Indeed, if $b$ is even, then $r^{\flat}=1/b\leq a/b$; if $b$ is odd, then
$r^{\flat}=\frac{1}{2b}\leq a/b$.  Theorem~\ref{thm:support-dominance}, with
$u=r^{\flat}$ and $d=L$, now gives \eqref{eq:core-dominates}.

If Conjecture~\ref{conj:HX} holds, then \eqref{eq:unit-family} follows by
specialization.  Conversely, suppose \eqref{eq:unit-family} holds.  Since $L$
is even, write $L=2m$.  Then
\[
  \T_{\frac{1}{2},k}(q)
  \geqcoef \T_{\frac{1}{2m},k}(q)
  =\T_{r^{\flat},k}(q)
  \geqcoef \T_{r,k}(q),
\]
where the last inequality is \eqref{eq:core-dominates}.  This proves the
equivalence.

It remains to justify the finite range \eqref{eq:finite-unit-family}.  If
$2m>k$, then the only multiple of $2m$ in $\{0,1,\ldots,k\}$ is zero, so
Proposition~\ref{prop:trace-formula} gives
\begin{equation}\label{eq:large-denominator-baseline}
  \T_{\frac{1}{2m},k}(q)=\Phi_{k,0}(q)=\frac{1}{(q;q)_k}.
\end{equation}
For $r=1/2$, the same proposition gives
\begin{equation}\label{eq:half-explicit}
  \T_{\frac{1}{2},k}(q)
  =\sum_{j=0}^{\lfloor k/2\rfloor}
      q^{2j(j+1)}\Phi_{k,2j}(q)
  \geqcoef \Phi_{k,0}(q),
\end{equation}
where the equality is \cite[eq.\ (19)]{HanXiong} and the dominance holds
because every summand has nonnegative coefficients.
Thus, every case with $2m>k$ is automatic, and the case $m=1$ of
\eqref{eq:unit-family} is an equality.  Only
the range \eqref{eq:finite-unit-family} can be nontrivial.
\end{proof}

As a first consequence, the reduction gives a second proof of
\cite[Proposition 4.1]{HanXiong}, which Han and Xiong obtained by direct
computation with $k=1,2,3$.

\begin{corollary}\label{cor:small-k}
Conjecture~\ref{conj:HX} holds for every positive rational $r$ when
$k=1,2,3$.
\end{corollary}

\begin{proof}
For $k\leq 3$, there is no integer $m$ satisfying
$2\leq m\leq\lfloor k/2\rfloor$.  Theorem~\ref{thm:canonical-reduction}
therefore leaves no nontrivial primitive case.
\end{proof}

For $k=4$ and $k=5$, the reduction leaves only $r=1/4$; the $k=4$
instance is worked out in Example~\ref{ex:k4}.  For $k=6$ and $k=7$, it
leaves $r=1/4$ and $r=1/6$.  This illustrates why the even-core family is the
natural one for further work.

\subsection{Proof of the finite verification}
\label{subsec:proof-verification}

By Lemma~\ref{lem:cumulative}, the proof of
Theorem~\ref{thm:finite-verification} reduces to checking the finite list
of cumulative inequalities
\eqref{eq:cumulative-sums} for every pair $(k,m)$ in the range
\eqref{eq:finite-unit-family}.  This is the analogue, for the whole
primitive family, with the pairing replaced by
\eqref{eq:cumulative-sums}, of the verification carried out for $r=1/4$ and
$r=1/3$ in \cite[Section~6]{HanXiong}, whose code and data are posted at
\url{https://irma.math.unistra.fr/~guoniu/halfconj.html}.  The verification
program constructs all Gaussian polynomials by the recurrence
\begin{equation}\label{eq:qbinomial-recurrence}
  \qbinom{k}{s}
  =\qbinom{k-1}{s}
   +q^{k-s}\qbinom{k-1}{s-1},
\end{equation}
with the usual boundary values.  At each stage it stores every Gaussian
polynomial as its complete integer coefficient vector, whose degree is
$s(k-s)$.  It then constructs $A_{1,k}$ and $A_{m,k}$ from
\eqref{eq:A-m-k}, forms the complete polynomial $D_{m,k}$, and checks every
cumulative sum in \eqref{eq:cumulative-sums}.  The number of nontrivial pairs
is
\[
  \sum_{k=1}^{200}
  \max\!\left(\left\lfloor\frac{k}{2}\right\rfloor-1,0\right)
  =\sum_{j=1}^{99}j+\sum_{j=1}^{98}j
  =9{,}801.
\]

Running the program on all of these pairs produced no violation of
\eqref{eq:cumulative-sums}.  We record the outcome as a proposition, since it
is the one input to Theorem~\ref{thm:finite-verification} that is not proved
by hand.

\begin{proposition}[Output of the exact computation]\label{prop:computer-output}
For every pair $(k,m)$ satisfying
\[
  1\leq k\leq 200,
  \qquad
  2\leq m\leq\left\lfloor\frac{k}{2}\right\rfloor,
\]
the series $D_{m,k}(q)/(1-q)$ has nonnegative integer coefficients.
\end{proposition}

\begin{proof}
The computation described above was run for all $9{,}801$ admissible pairs.
For each pair, the program scanned the complete coefficient vector of
$D_{m,k}$ from lowest to highest degree and verified that every running sum
in \eqref{eq:cumulative-sums} is nonnegative.  No negative running sum
occurred.  Since the recurrence \eqref{eq:qbinomial-recurrence} and the
subsequent scans use only exact integer arithmetic, this establishes the
stated finite result.
\end{proof}

\begin{proof}[Proof of Theorem \ref{thm:finite-verification}]
Fix $1\leq k\leq 200$.  Proposition~\ref{prop:computer-output} and
Lemma~\ref{lem:cumulative} establish \eqref{eq:unit-family} for every
nontrivial $m$ in the finite range \eqref{eq:finite-unit-family}.
Theorem~\ref{thm:canonical-reduction} then gives the desired inequality for
every positive rational $r$.
\end{proof}

\subsection{The remaining problem}\label{subsec:remaining}

By Theorem~\ref{thm:canonical-reduction} and the trace formula
\eqref{eq:trace-formula-j}, what remains of Conjecture~\ref{conj:HX} is the
uniform coefficientwise inequality
\begin{equation}\label{eq:remaining-open-family}
 \sum_{j=0}^{\lfloor k/2\rfloor}
   \frac{q^{2j(j+1)}}{(q;q)_{2j}(q;q)_{k-2j}}
 \geqcoef
 \sum_{\ell=0}^{\lfloor k/(2m)\rfloor}
   \frac{q^{2m^2\ell^2+(m+1)\ell}}
        {(q;q)_{2m\ell}(q;q)_{k-2m\ell}}
\end{equation}
for all $m\geq 2$ and $k\geq 1$. 
 By Theorem~\ref{thm:finite-verification},
\eqref{eq:remaining-open-family} holds for every $m\geq 2$ when $k\leq 200$;
what is at issue is uniformity in $k$.  Neither the support-dominance theorem
nor the finite computation supplies an all $k$ proof of
\eqref{eq:remaining-open-family}.
 In particular, a uniform all $k$ proof
for $m=2$, corresponding to $r=1/4$, would already settle the first unresolved
primitive family.  Han and Xiong single out precisely this case as
\cite[Conjecture 2.3]{HanXiong}, on the grounds of the proximity of $1/4$ to
$1/2$. Theorem~\ref{thm:canonical-reduction} gives a second reason for its
prominence: it is the first nontrivial primitive parameter in the canonical
family, equivalently the largest primitive parameter below $1/2$.

\appendix
\section{AxiomProver's autonomous Lean formalization}\label{sec:AI}

AxiomProver is an AI system for mathematical research via formal proof that is
currently under development.  As an early test case, we treated
Theorem~\ref{thm:support-dominance}, Corollary~\ref{cor:half-line}, and
Theorem~\ref{thm:canonical-reduction} as an end-to-end formalization target:
the formal statements and their proofs were produced by AxiomProver from a
natural-language statement of the problem that contained no proofs, and were
then verified by the Lean proof assistant.  Using that formal development as a
reference, the human author wrote the exposition in the main text for a
mathematical audience.  Readers not interested in automated theorem proving may
skip this appendix.

\subsection*{Scope of the formalization}
The formalization concerns three of the main results of this paper, stated in
Lean for the Gaussian binomial $\qbinom{k}{s}$ defined by the $q$-Pascal
recursion, the partition counts $p_m(n)$, the exponent $E_r(s)$ of
\eqref{eq:numerator-expansion}, the integer support of $\T_{r,k}$, and the
trace $\T_{r,k}$ itself, taken there in the convolution form given by
Proposition~\ref{prop:trace-formula}.  The formalized statements are
Theorem~\ref{thm:support-dominance}, Corollary~\ref{cor:half-line}, and
Theorem~\ref{thm:canonical-reduction}.
Theorem~\ref{thm:finite-verification} was not part of the Lean
formalization.

\subsection*{Process}
The formal proofs were developed and verified using Lean~4.31
together with Mathlib and AxiomLib, an internal Lean library at Axiom Math
which supplies the $q$-Pochhammer symbol and which was uploaded to the system
along with the inputs listed below; compatibility with earlier or later
versions is not guaranteed, owing to the evolving nature of the Lean~4 compiler
and its core libraries.  The proofs contain no \texttt{sorry}
declarations and introduce no additional axioms.  The human author checked
that the statements in \texttt{problem.lean} express
Theorem~\ref{thm:support-dominance}, Corollary~\ref{cor:half-line}, and
Theorem~\ref{thm:canonical-reduction} as stated in the main text.  All files
needed to inspect and rerun the formalization are posted in the repository
\begin{center}
  \url{https://github.com/AxiomMath/QBinomialTrace}
\end{center}
The input supplied to AxiomProver consisted of
\begin{itemize}
\item \texttt{main.tex}, a self-contained natural-language statement of
  the definitions and of the three target results;
\item \texttt{HanGuoNiu.pdf}, a copy of \cite{HanXiong}; and
\item \texttt{task.md}, which instructs the system to read the two
  preceding files and to formalize and verify the three named results.
\end{itemize}
From these inputs, AxiomProver autonomously produced
\begin{itemize}
\item \texttt{problem.lean}, a Lean formalization of the problem statement; and
\item \texttt{solution.lean}, a complete Lean formalization of the proof.
\end{itemize}

After AxiomProver generated the solution, the human author wrote this paper for
human readers.  A research paper is a narrative designed to communicate ideas to
people, whereas a Lean file is written to satisfy a proof-checking kernel.
At first glance, therefore, the formal proofs do not resemble the narrative
presented here.

\section*{Declaration of generative AI and AI-assisted technologies in the
manuscript preparation process}\noindent
As described in the preceding appendix, AxiomProver was used to produce and
formally verify, in Lean, the proofs of
Theorem~\ref{thm:support-dominance}, Corollary~\ref{cor:half-line}, and
Theorem~\ref{thm:canonical-reduction}.  The exposition was subsequently written
by the human author.


\begin{thebibliography}{99}

\bibitem{Andrews}
G.~E. Andrews,
\emph{The Theory of Partitions},
Cambridge Mathematical Library, Cambridge University Press, Cambridge, 1998.

\bibitem{GasperRahman}
G.~Gasper and M.~Rahman,
\emph{Basic Hypergeometric Series},
Encyclopedia of Mathematics and its Applications, vol.~35,
Cambridge University Press, Cambridge, 1990.

\bibitem{HanXiong}
G.-N. Han and H.~Xiong,
The $1/2$-conjecture for $q$-binomial coefficients with fractional index,
\emph{arXiv preprint} arXiv:2606.01919v1 (2026).

\bibitem{Stanley}
R.~P. Stanley,
\emph{Enumerative Combinatorics, Vol.~1}, second ed.,
Cambridge Studies in Advanced Mathematics, vol.~49,
Cambridge University Press, Cambridge, 2012.

\end{thebibliography}
\end{document}